\documentclass[]{interact}

\usepackage{epstopdf}
\usepackage[caption=false]{subfig}
\usepackage{mathrsfs}
\usepackage[algoruled]{algorithm2e}
\usepackage{algorithmic}
\usepackage[numbers,sort&compress]{natbib}
\bibpunct[, ]{[}{]}{,}{n}{,}{,}
\renewcommand\bibfont{\fontsize{10}{12}\selectfont}
\makeatletter
\def\NAT@def@citea{\def\@citea{\NAT@separator}}
\makeatother

\theoremstyle{plain}
\newtheorem{theorem}{Theorem}[section]
\newtheorem{lemma}[theorem]{Lemma}

\newtheorem{assumption}{Assumption}
\theoremstyle{definition}
\newtheorem{definition}[theorem]{Definition}

\theoremstyle{remark}
\newtheorem{remark}{Remark}

\begin{document}


\title{A descent subgradient method using Mifflin's line search for nonsmooth nonconvex optimization}

\author{
\name{Morteza Maleknia\textsuperscript{a}\thanks{Email addresses: maleknia.morteza@gmail.com (M. Maleknia), m.soleimani.d@ut.ac.ir (M. Soleimani-damaneh) } and Majid Soleimani-damaneh\textsuperscript{a} \thanks{}}
\affil{\textsuperscript{a}School of Mathematics, Statistics and Computer Science, College of Science, University of
	Tehran. 
	}
}

\maketitle

\begin{abstract}
We propose a descent subgradient algorithm for minimizing a function $f:\mathbb{R}^n\to\mathbb{R}$, assumed to be locally Lipschitz, but not necessarily smooth or convex. To find an effective descent direction, the Goldstein $\varepsilon$-subdifferential is approximated through an iterative process. The method enjoys a new two-point variant of Mifflin's line search in which the subgradients are arbitrary.
Thus, the line search procedure is easy to implement. Moreover, in comparison to bundle methods, the quadratic subproblems have a simple structure, and to handle nonconvexity the proposed method requires no algorithmic modification.  We study the global convergence of the method and prove that any accumulation point of the generated sequence  is Clarke stationary, assuming that the objective $f$ is weakly upper semismooth. We illustrate the efficiency and effectiveness of the proposed algorithm on a collection of academic and semi-academic test problems.
\end{abstract}


\begin{amscode}
	65K05, 65K10, 90C26
	 \end{amscode}

\section{Introduction}

In this study, we concentrate on the following unconstrained minimization problem:
\begin{equation}\label{Main-Problem}
	\min \,\, f(\bm x) \quad \text{s.t.} \quad \bm x\in\mathbb{R}^n,
\end{equation}
where $f:\mathbb{R}^n\to\mathbb{R}$ is locally Lipschitz, but not necessarily smooth or convex. Such
problems arise in many applied fields such as matrix algebra, approximation theory, optimal control, image processing, and data analysis \cite{Poly-App,Ahoo, Makela_book,Cond1,image}. Over more than four decades, a great deal of effort has gone into developing effective algorithms for solving problem~\eqref{Main-Problem}. In the following, we concisely review the most important relevant methodologies.
\subsection{Literature review}
Bundle methods, originally developed by Lemar\'echal  \cite{lemarchal-book1,lemarchal-book2,Lemar-bundle} and Wolfe  \cite{Cojugate-Wolf}, are of the most common tools for solving problem \eqref{Main-Problem}. A well-developed theoretical base and a nice practical performance make these  methods highly popular in nonsmooth optimization. Bundle methods store a number of previously computed trial points along with
the corresponding subgradients into a bundle of information. Using the elements of this bundle, a model function for the objective function is constructed. As a standard manner, by minimizing the model function, one can obtain a search direction. Next, a line search procedure finds the next trial point, and the bundle of information is
updated accordingly. The aggregation strategy proposed by Kiwiel  \cite{Aggregate-sub} was an important contribution to the field in resolving some difficulties with the amount of required storage. One can point to the proximal bundle method \cite{Proximal-Bundle,BT-method} as one of the most efficient variants of the bundle methods. These methods keep the model function local enough by means of a proximity parameter. Variable metric bundle methods \cite{Limited-bundle,v-metric1,v-metric33,v-metric} employ quasi-Newton techniques to augment the model function with an approximation of the Hessian matrix. Moreover, some recent variants of the bundle methods that deal with approximate subgradients can be found in \cite{inexactBundle3,inexactBundle1,inexactBundle2,Hosseini1}. For more recent developments in bundle methods and their applications one can refer to \cite{Saga1,Noll1,Noll2,Gaudioso,Gaudiso2,Saga2}. One drawback of bundle methods is that their generalization from convex to nonconvex case requires serious algorithmic modifications, which leads to a much less satisfactory numerical performance.

In 2002, Burke et al. \cite{GS-performance} initiated a giant stride towards approximating the subdifferential set by sampling gradients. The results of that work led to proposing an implementable algorithm, namely Gradient Sampling (GS) \cite{Burke2005}. This method approximates the $\varepsilon$-steepest descent direction to obtain a search direction during each iteration. Then, it employs a standard backtracking Armijo line search to find a suitable step size. A work of Kiwiel \cite{Kiwiel2007}  improved the convergence results of the original GS method. An extension of the method for solving constrained problems was presented in  \cite{Curtis2012}. A specific variant of the GS approach for solving min-max problems was appeared in \cite{Fast-GS}. Although the original GS approach is robust, it requires $m>n$ gradient evaluations during each iteration, which makes the method computationally expensive. To tackle this difficulty, some variants of the GS method with the aim of reducing the number of gradient evaluations were developed in \cite{Curtis2013,Maleknia-Coap,Maleknia-Oms}. Moreover, some special types of the GS approach, that inexactly solve the corresponding quadratic subproblems, can be found in \cite{Curtis-QP,Maleknia-jota}.

As another class of methods that can deal with problem \eqref{Main-Problem}, one can point to the subgradient methods originally proposed by Shor  \cite{shorbook}. The initial subgradient method has a very simple structure as any direction  opposite to an arbitrary subgradient can be employed as a search direction, not to mention making use of an off-line sequence of step sizes. However, the approach suffers from several limitations, including poor speed of convergence, lack of descent, deficiency of a practical stopping criterion based on first-order optimality conditions, and limited convergence results for nonconvex objectives. To boost the convergence speed, a subgradient method with space dilation was suggested in \cite{Space-dilation1}. Based on the space dilation operator, Shor developed another variant of subgradient methods, namely $r$-algorithm \cite{shorbook,r-algorithm}. One can consider these modified subgradient methods as variable metric methods which do not satisfy the secant equation. Besides a limited theoretical foundation, these approaches do not have a stopping criterion based on a necessary optimality condition. Moreover, the amount of required storage for storing the corresponding operators poses some difficulties with medium and large-scale problems. Owing to some features of bundle methods, Bagirov et al. \cite{bagirov2010,bagirov2012} proposed a descent subgradient algorithm for solving problem \eqref{Main-Problem}. Their approach is interesting as it enjoys a practical stopping criterion and, unlike bundle methods, it requires no algorithmic modifications to handle nonconvexity. However, the user has to supply those subgradients which approximately satisfy the conditions in the Lebourg's mean value theorem \cite{Clarke1990}, namely quasi-secants. In fact, Bagirov et al.'s method does not work with arbitrary subgradients. Another descent subgradient algorithm in which the subgradients are not arbitrary can be found in \cite{m-amiri}.

\subsection{The proposed method}
In this study, we propose a descent subgradient algorithm for solving problem~\eqref{Main-Problem}. By Rademacher's theorem \cite{Evans2015}, we know that the locally Lipschitz function $f:\mathbb{R}^n\to\mathbb{R}$ is differentiable almost everywhere on $\mathbb{R}^n$. Moreover, in many practical situations, locally Lipschitz functions are continuously differentiable (smooth) almost everywhere on $\mathbb{R}^n$. Indeed, while minimizing a locally Lipschitz function over $\mathbb{R}^n$ using a machine which uses \textbf{IEEE} double or single precision arithmetic, a nonsmooth point is never encountered except in trivial or pathological cases. In this regard, the Clarke subdifferential merely suggests the traditional steepest descent direction as a search direction, which is not an effective descent direction in nonsmooth optimization \cite{Asl,Burke2021}. To avoid this issue, we employ the Goldstein $\varepsilon$-subdifferential  \cite{Gold1} which stabilizes our choice of the search direction. More precisely, our main idea is to develop an iterative procedure to approximate the Goldstein $\varepsilon$-subdifferential, which leads to an estimation of the $\varepsilon$-steepest descent direction.

The heart of the proposed method is a new two-point variant of the Mifflin's line search whose finite convergence is guaranteed under the  assumption that the objective $f$ is weakly upper semismooth. Thanks to the proposed line search, our algorithm works with arbitrary subgradients, which is not the case in \cite{bagirov2012} and \cite{m-amiri}. As opposed to bundle methods, the proposed method requires no algorithmic and parametric modifications to handle nonconvexity. In addition, the structure of the quadratic subproblems is simpler than the bundle type methods. In contrast with original GS method, our approximation of the Goldstein $\varepsilon$-subdifferential is improved sequentially, and hence the proposed approach needs fewer subgradient evaluations than the original GS method. To control the size of quadratic subproblems, the user can optionally employ an adaptive subgradient selection strategy to discard almost redundant subgradients. 

We study the global convergence of the method and prove that any accumulation point of the generated sequence is Clarke stationary for objective $f$. By means of numerical experiments, we show the efficiency of the method in practice. To this end, first we consider a set of academic nonsmooth convex and nonconvex test problems to provide some comparative results. Next, we apply our method to a nonsmooth model arising in data clustering. In our third experiment, we consider the problem of Chebyshev approximation by polynomials. Finally, we turn to  the problem of minimizing eigenvalue products.

\subsection{Outline}
In Section \ref{Sec2}, we provide some required preliminaries. Section \ref{Sec3} describes the proposed approach for finding a descent direction. Approximate Clarke stationary points are computed in Section \ref{Sec4}, and a Clarke stationary point for objective function $f$ is obtained in Section \ref{Sec5}. Numerical results are reported  in Section \ref{Sec6}, and Section \ref{Sec7} concludes the paper.

\section{Preliminaries}\label{Sec2}
Throughout this paper, we use the following notations. The usual inner product in the Euclidean space $\mathbb{R}^n$  is denoted by $\bm x^T\bm y$, which induces the Euclidean norm $\lVert\bm x\rVert=(\bm x^T\bm x)^{1/2}$. An open ball with center $\bm x\in \mathbb{R}^n$ and radius $\varepsilon\geq0$ is denoted by $\mathcal B(\bm x, \varepsilon)$, that is,
$$\mathcal B(\bm x, \varepsilon):=\{\bm y\in\mathbb{R}^n \,\,:\,\, \lVert \bm y-\bm x\rVert< \varepsilon\}. $$
Moreover, $\mathbb{N}$ is the set of natural numbers, $\mathbb{N}_0:=\mathbb{N}\cup\{0\}$, and $\mathbb{R}_+:=(0,\infty).$

Suppose $f:\mathbb{R}^n\to\mathbb{R}$ is a locally Lipschitz function. Then, by Rademacher's theorem~\cite{Evans2015}, $f$ is differentiable almost everywhere on $\mathbb{R}^n$. Let
$$\Omega_f:= \{\bm x\in\mathbb{R}^n \,\, : \,\, f\,\, \text{is not differentiable at}\,\, \bm x  \}.$$
Then,  the Clarke subdifferential of $f$ at a point $\bm x\in\mathbb{R}^n$ is defined as \cite{Clarke1990}
\begin{equation*}
\partial f(\bm x):= \texttt{conv} \{\boldsymbol{\xi}\in\mathbb{R}^n \,\,:\,\, \exists \, \{\bm x_i \}\subset\mathbb{R}^n \setminus\Omega_f\,\,\,\, \text{s.t.} \,\,\,\, \bm x_i\to \bm x \,\, \text{and} \,\, \nabla f(\bm x_i)\to \boldsymbol{\xi} \},
\end{equation*}
where $\texttt{conv}$ denotes the convex hull operator. Furthermore, for any $\varepsilon\geq 0,$ the (Goldstein) $\varepsilon$-subdifferential of $f$ at a point $\bm x\in\mathbb{R}^n$ is the set  \cite{Bagirov2014}
$$\partial_\varepsilon f(\bm x):=\texttt {cl\,conv}\{\partial f(\bm y) \,\, : \,\, \bm y\in \mathcal{B}(\bm x, \varepsilon) \},$$
in which $\texttt{cl\,conv}$ is the closure of the convex hull. If $\varepsilon=0$, we have $\partial f(\bm x)=\partial_0 f(\bm x)$, for all $\bm x\in\mathbb{R}^n$. In addition, for any $\varepsilon\geq 0$ and $\bm x\in\mathbb{R}^n$, the set $\partial_\varepsilon f(\bm x)$ is a nonempty, convex and compact subset of $\mathbb{R}^n$. If $f$ is differentiable at $\bm x\in\mathbb{R}^n$, then $\nabla f(\bm x)\in\partial f(\bm x)$. Furthermore, If $f$ is smooth at $\bm x\in\mathbb{R}^n$, we have $\{\nabla f(\bm x)\}=\partial f(\bm x)$.  Also, the set-valued map $\partial_\varepsilon f:\mathbb{R}^n\rightrightarrows\mathbb{R}^n$ is locally bounded and upper semicontinuous \cite{Clarke1990}. It is recalled that for a point $\bm x\in\mathbb{R}^n$ to be a local minimizer of the locally Lipschitz function $f$, it is necessary that $\bm 0\in\partial f(\bm x)$. Such a point is called a \emph{Clarke stationary} point.

In the proposed method, the following concept of stationarity plays a crucial role.

\begin{definition}\label{Def1}
	Let $\bm x\in\mathbb{R}^n$, $\varepsilon>0$, and  $\delta>0$. Assume $\mathcal G_\varepsilon({\bm x})\subset \partial_\varepsilon f(\bm x)$ is a nonempty inner approximation of $\partial_\varepsilon f(\bm x)$.‌Then the point $\bm x\in\mathbb{R}^n$ is called a $(\delta, \mathcal G_\varepsilon({\bm x}))$-stationary point if
	$$\min\{\lVert \bm g\rVert \,\, : \,\, \bm g\in\text{\rm\texttt{conv}} \mathcal{G}_\varepsilon(\bm x)  \}\leq \delta. $$
\end{definition}

\section{In quest of a descent direction}\label{Sec3}
Our main idea to obtain a descent direction for the locally Lipschitz function $f:\mathbb{R}^n\to\mathbb{R}$ at a point $\bm x\in\mathbb{R}^n$ is to approximate the $\varepsilon$-steepest descent direction. In this respect, we concisely review the notions of steepest descent and  $\varepsilon$-steepest descent directions.

For the locally Lipschitz objective $f$, the steepest descent direction at a point $\bm x\in\mathbb{R}^n$ is obtained from the following minimization problem \cite{Bonnansbook,Burke2005}:
\begin{equation}\label{steepest-descent}
\min\{ \lVert \boldsymbol{\xi}\rVert \,\, : \,\,  \boldsymbol{\xi}\in\partial f(\bm x) \}.
\end{equation}
Let $\boldsymbol{\xi}^*\neq \bm 0$ be the optimal solution of problem \eqref{steepest-descent}. Then, the direction $\bar{\bm{d}}:=-\boldsymbol{\xi}^*/\lVert \boldsymbol{\xi}^*\rVert$ is called \emph{(normalized) steepest descent direction}. In many iterative algorithms for solving problem \eqref{Main-Problem}, we often land on a continuously differentiable point which is close by to the nonsmooth region $\Omega_f$. In this situation, $\partial f(\bm x)$ does not contain any information of the nearby nonsmooth region; in other words, $\partial f(\bm x)=\{\nabla f(\bm x) \}$. Thus, the steepest descent direction coincides with the direction $-\nabla f(\bm x)/\lVert \nabla f(\bm x) \rVert$, which is not an effective descent direction for nonsmooth function $f$. In contrast, in the same situation, the $\varepsilon$-subdifferential $\partial_\varepsilon f(\bm x)$ can capture some local information of the nearby nonsmooth region and provide an effective descent direction. If $\boldsymbol{\xi}^*\neq \bm 0$ solves the following minimization problem:
\begin{equation}\label{eps-steepest-descent}
\min\{ \lVert \boldsymbol{\xi} \rVert \,\, : \,\,   \boldsymbol{\xi}\in\partial_\varepsilon f(\bm x)\},
\end{equation}
we call the direction $\tilde{\bm{d}}:=-\boldsymbol{\xi}^*/\lVert \boldsymbol{\xi}^*\rVert$  \emph{(normalized) $\varepsilon$-steepest descent direction}, which is similar to those introduced in \cite{Burke2005,shorbook}. As observed, to solve problem \eqref{eps-steepest-descent}, we need to know the entire subdifferential on $\mathcal{B}(\bm x, \varepsilon)$, which is impractical in many real-life situations. In this regard, we develop an iterative procedure to efficiently  approximate $\partial_\varepsilon f(\bm x)$.

For a given point $\bm x\in\mathbb{R}^n$, and given scalars $m\in\mathbb{N}$ and $\varepsilon>0$, let
$$\mathcal G_\varepsilon(\bm x):=\{\boldsymbol{\xi}_1, \boldsymbol{\xi}_2,\ldots,\boldsymbol{\xi}_m\}\subset \partial_\varepsilon f(\bm x)$$
be a collection of subgradients. Then, we consider $$\texttt{conv}\, \mathcal G_\varepsilon(\bm x)\subset \partial_\varepsilon f(\bm x)$$
as an inner approximation of $\partial_\varepsilon f(\bm x)$, and solve the following minimization problem:
\begin{equation}\label{Approx-eps-d-d}
\min\{\lVert\bm g\rVert \,\, : \,\, \bm g\in \texttt{conv}\, \mathcal{G}_\varepsilon(\bm x)   \},
\end{equation}
which is a practical approximation of problem \eqref{eps-steepest-descent}. If $\bm g^*\neq \bm 0$ is the optimal solution of problem \eqref{Approx-eps-d-d}, the direction ${\bm d}:=-\bm g^*/\lVert \bm g^* \rVert$ is an approximation of $\tilde{\bm d}$.
In case $\texttt{conv} \mathcal{G}_\varepsilon(\bm x)$ is a good approximation of $\partial_\varepsilon f(\bm x)$, one can use the direction ${\bm d}$ to take a descent step, i.e., there exists the step length $t>0$ satisfying the following sufficient decrease condition:
\begin{equation}\label{Suff-decrease}
f(\bm x+t {\bm d})-f(\bm x)\leq -\beta t \lVert \bm g^* \rVert \quad \text{and} \quad t\geq \bar t,
\end{equation}
where $\beta\in(0,1)$ is a sufficient decrease parameter, and $\bar t>0$ is a lower bound for the step length $t$. Otherwise, the working set $\mathcal{G}_\varepsilon(\bm x)$ should be improved by appending a new element of $\partial_\varepsilon f(\bm x)$, namely $\boldsymbol{\xi}_{m+1}$. The new subgradient $\boldsymbol{\xi}_{m+1}\in \partial_\varepsilon f(\bm x)$ must be chosen such that
\begin{equation}\label{Criteria-0}
\boldsymbol{\xi}_{m+1}\notin\texttt{conv} \mathcal{G}_\varepsilon(\bm x).
\end{equation}
In this way, if we update $\mathcal{G}_\varepsilon(\bm x)$ by
$$\mathcal{G}^+_\varepsilon(\bm x):=\mathcal{G}_\varepsilon(\bm x)\cup \{\boldsymbol{\xi}_{m+1} \},$$
we have $\texttt{conv}\mathcal{G}_\varepsilon(\bm x)\subsetneq\texttt{conv}\mathcal{G}^+_\varepsilon(\bm x)$. In other words, our approximation of $\partial_\varepsilon f(\bm x)$ is improved significantly. The following lemma provides a useful criterion to find the new subgradient $\boldsymbol{\xi}_{m+1}\in \partial_\varepsilon f(\bm x)$ which satisfies condition \eqref{Criteria-0}.
\begin{lemma}
	Let $\bm g^*\neq \bm 0$ be the optimal solution of problem \eqref{Approx-eps-d-d}, and ${\bm d}=-\bm g^*/\lVert \bm g^* \rVert$. For a $\beta\in(0,1)$ and $\boldsymbol{\xi}_{m+1}\in\partial_\varepsilon f(\bm x)$, assume
	\begin{equation}\label{Criterion1}
	\boldsymbol{\xi}_{m+1}^T {\bm d} \geq -\beta \lVert \bm g^*\rVert.
	\end{equation}
	Then $\boldsymbol{\xi}_{m+1}\notin\text{\rm\texttt{conv}} \mathcal{G}_\varepsilon(\bm x)$.
\end{lemma}
\begin{proof}
	Since $\bm g^*$ solves problem \eqref{Approx-eps-d-d}, we have \cite{Rockafellar2004}
	$$-\bm g^*\in \mathcal{N}_{\texttt{conv} \mathcal{G}_\varepsilon(\bm x)} (\bm g^*), $$
	in which $\mathcal{N}_{\texttt{conv} \mathcal{G}_\varepsilon(\bm x)} (\bm g^*)$ denotes the normal cone of the set $\texttt{conv} \mathcal{G}_\varepsilon(\bm x)$ at $\bm g^*$. This means
	\begin{equation}\label{L1-1}
	\bm g^T \bm g^*\geq \lVert \bm g^*\rVert^2, \quad \text{for all} \,\, \bm g\in \texttt{conv} \mathcal{G}_\varepsilon(\bm x).
	\end{equation}
	Therefore, if we compute $\boldsymbol{\xi}_{m+1}\in\partial_\varepsilon f(\bm x)$ such that $\boldsymbol{\xi}_{m+1}^T {\bm d} \geq -\beta \lVert \bm g^*\rVert$, inequality \eqref{L1-1} implies $\boldsymbol{\xi}_{m+1}\notin\text{\rm\texttt{conv}} \mathcal{G}_\varepsilon(\bm x)$.
\end{proof}
Based on the preceding discussion, in Algorithm \ref{Line-Search},  we present a new two-point variant of Mifflin's line search \cite{kiwielbook,Mifflin2} which either obtains the step length $t>0$ satisfying sufficient decrease condition \eqref{Suff-decrease} or provides some $\boldsymbol{\xi}_{m+1}\in\partial_\varepsilon f(\bm x)$ satisfying criterion~\eqref{Criterion1}.

\begin{algorithm}
	\caption{A Two-Point Line Search (LS)}
	\label{Line-Search}
	\hspace*{\algorithmicindent}\textbf{Inputs:} Radius $\varepsilon\in(0,1)$, current point $\bm x\in\mathbb{R}^n$, and search direction ${\bm d}=-\bm g^*/\lVert \bm g^*\rVert$ with $\bm g^*\neq \bm 0$. \\
	\hspace*{\algorithmicindent}\textbf{Parameters:} $\beta_1, \beta_2\in(0,1)$ with $\beta_1<\beta_2$, reduction factor $\zeta\in(0,0.5)$, lower bound $\bar t\in(0, \varepsilon)$, and positive integer $p\in\mathbb{N}$. \\
	\hspace*{\algorithmicindent}\textbf{Outputs:} A point $\bm s\in\mathbb{R}^n$ and the indicator ${\rm \bm I}\in\{0,1\}$.\\
	
	\hspace*{\algorithmicindent}\textbf{Function:} $\{\bm s, {\rm \bm I}\}$\,\,=\,\,\texttt{T-PLS}\,($\varepsilon$, $\mathbf x$, ${\bm d}$)
	
	\begin{algorithmic}[1]
		\STATE{\textbf{Initialization:} Choose $t_0\in(\bar t, \varepsilon)$ and set $\bar t_0:=1$. Compute $\boldsymbol{\xi}_0\in\partial f(\bm x+t_0{\bm d})$, and  set $t^l_0:=0, t^u_0:=\varepsilon$, $i:=0$ ;}
		\WHILE{ true }
		
		\IF{$f(\bm x+t_i{\bm d})-f(\bm x)\leq -\beta_1\, t_i\, \lVert \bm g^*\rVert,$}
		\STATE {Set $t^l_{i+1}:=t_i, \,\,\,\, t^u_{i+1}:=t^u_i$ ;}
		\ELSE
		\STATE{Set $t^l_{i+1}:=t^l_i, \,\,\,\, t^u_{i+1}:=t_i$ ;}
		\ENDIF
		\IF{$f(\bm x+\bar t_i{\bm d})-f(\bm x)\leq -\beta_1\, \bar t_i\, \lVert \bm g^*\rVert,$ \rm{\textbf{and}} $ \bar t_i\geq\bar{t}\,\,,$}
		\STATE {Set ${\rm \bm I}:=1$ and $\bm s:=\bm x+\bar t_i{\bm d}$ ;}
		\RETURN {$\{\bm s, {\rm \bm I}\}$ and \textbf{Stop} ;}
		\ENDIF
		\IF{$\boldsymbol{\xi}_i^T{\bm d}\geq -\beta_2 \lVert \bm g^*\rVert ,$}
		\STATE {Set ${\rm \bm I}:=0$ and $\bm s:=\boldsymbol{\xi}_i$ ;}
		\RETURN {$\{\bm s, {\rm \bm I}\}$ and \textbf{Stop} ;}
		\ENDIF
		\STATE{Choose $t_{i+1}\in \left[ t^l_{i+1}+\zeta(t^u_{i+1}-t^l_{i+1}),\, t^u_{i+1}-\zeta(t^u_{i+1}-t^l_{i+1})   \right]$  ;}
		\STATE{Set $\bar t_{i+1}:=\exp(\frac{\log t_0}{p})^{i+1}$ ;}
		\STATE{Compute $\boldsymbol{\xi}_{i+1}\in \partial f(\bm x+t_{i+1} {\bm d})$ ;}
		\STATE{Set $i:=i+1$ ;}
		
		\ENDWHILE
	\end{algorithmic}
	\hspace*{\algorithmicindent}\textbf{ End Function}
\end{algorithm}

There are three conditional blocks and two step lengths in Algorithm \ref{Line-Search}, $t_i$ and $\bar t_i$. We employ the step length $t_i$ to find an element of $\partial_\varepsilon f(\bm x)$ satisfying \eqref{Criterion1}, which is controlled by the third conditional block. Since we should keep the computed subgradients in $\partial_\varepsilon f(\bm x)$, the trial step length $t_i$ varies within the interval $(0,\varepsilon)$. The length of this interval is efficiently reduced by the first conditional block. Also, by using the trial step length $\bar t_i$, we look for a suitable step length in the interval $(0,1]$  satisfying the sufficient decrease condition \eqref{Suff-decrease}, which is done in the second conditional block. Consequently, the indicator ${\rm \bm I=1}$ suggests a descent step by using the resulting point $\bm s\in\mathbb{R}^n$, i.e., the current point $\bm x$ is updated by $\bm x^+:=\bm s$. On the other hand,  the indicator ${\rm \bm I=0}$ reveals that one can use the obtained subgradient $\bm s\in\partial_\varepsilon f(\bm x)$ to improve  $\mathcal{G}_\varepsilon(\bm x)$, i.e., we set $\boldsymbol{\xi}_{m+1}:=\bm s$ and $\mathcal{G}_\varepsilon(\bm x)$ is updated by $\mathcal{G}^+_\varepsilon(\bm x):=\mathcal{G}_\varepsilon(\bm x)\cup\{\boldsymbol{\xi}_{m+1}\}$.

In the following, we show that Algorithm~\ref{Line-Search} terminates after finite number of iterations. To this end, we start with the following lemma.

\begin{lemma}\label{L2} Suppose that  Algorithm \ref{Line-Search} does not terminate. Then
	\begin{itemize}
		\item[(i)] For any $i\geq 0$, we have $t_i\in\{t_{i+1}^l, t_{i+1}^u  \}$. Moreover, for all $i\geq 1$,
		\begin{align}
		&0<t^u_{i+1}-t^l_{i+1}\leq (1-\zeta) (t^u_{i}-t^l_{i}). \label{L2-1}\\&
		0\leq t^l_i\leq t^l_{i+1}< t^u_{i+1}\leq t^u_i\leq \varepsilon. \label{L2-2}
		\end{align}
		\item[(ii)] There exists $t^*\in[0,\varepsilon]$ such that $t^u_i\downarrow t^*$, $t^l_i\uparrow t^*$, and $t_i\to t^*$ as $i\to \infty$. In addition $$t^*\in \mathcal{T}:=\{t\,\, : \,\, f(\bm x +t {\bm d})-f(\bm x)\leq -\beta_1\, t\,  \lVert\bm g^*\rVert   \}.$$
		\item[(iii)] Let $\mathcal{I}:=\{ i\,\, : \, \, t^u_{i+1}=t_i \}$. Then $\mathcal{I}$ is infinite.
		
	\end{itemize}

\end{lemma}
\begin{proof}
	(i) This part follows immediately from the construction of the algorithm.
	
	(ii) We conclude from \eqref{L2-1} that the sequence $\{t^u_i-t^l_i\}_{i}$ is bounded from below and decreasing. Thus, it converges. Assume $\{t^u_i-t^l_i\}\to A$ as $i\to\infty$. Letting $i$ tend to infinity in inequality \eqref{L2-1}, we deduce $0\leq A\leq(1-\zeta)A$. Now, $\zeta\in(0,0.5)$ gives $A=0$. Since $\{t^u_i-t^l_i\}\to 0$ as $i\to\infty$, inequality \eqref{L2-2} implies the existence of $t^*\in[0,\varepsilon]$ such that $t^l_i\uparrow t^*, t^u_i\downarrow t^*$ as $i\to\infty$. Furthermore, the fact $t_i\in\{ t^l_{i+1}, t^u_{i+1} \}$ for all $i\geq 0$ yields $t_i\to t^*$ as $i\to\infty$. To prove $t^*\in \mathcal T$, we note that $t^l_i\in \mathcal T$ for all $i\geq 0$. In other words,
	$$f(\bm x+t^l_i {\bm d})-f(\bm x)\leq -\beta_1\, t^l_i\,\lVert\bm g^*\rVert, \qquad \text{for all} \,\, i\geq 0. $$
	Therefore, continuity of $f$ along with the fact that $t^l_i\uparrow t^*$ as $i\to\infty$  implies
	$$f(\bm x+t^* {\bm d})-f(\bm x)\leq -\beta_1\, t^*\,\lVert\bm g^*\rVert, $$
	which means $t^*\in \mathcal T.$

	(iii) First, we prove  $\mathcal{I}\neq\emptyset$. By contradiction, suppose $\mathcal{I}=\emptyset$, which means
	$$ f(\bm x+t_i {\bm d})-f(\bm x)\leq -\beta_1\, t_i\,  \lVert\bm g^*\rVert, \qquad \text{for all} \,\, i\geq 0. $$
	In particular, for $i=0$, we have
	\begin{equation}\label{L2-3}
	f(\bm x+t_0 {\bm d})-f(\bm x)\leq -\beta_1\, t_0\,  \lVert\bm g^*\rVert.
	\end{equation}
	On the other hand, at iteration $i=p-1$, one has $\bar t_{i+1}=\exp(\frac{\log t_0}{p})^p=t_0>\bar t$. This fact along with \eqref{L2-3} implies that Algorithm \ref{Line-Search} terminates at the iteration $i=p$  with indicator ${\rm \bm I=1}$, which violates the assumption.
	Thus, $\mathcal{I}\neq\emptyset$. Next, we prove  $\mathcal{I}$ is infinite. By contradiction, assume  $\mathcal{I}$ is finite. Then, as $\mathcal{I}\neq\emptyset$ and $t^u_i\downarrow t^*$ as $i\to\infty$, there exists $\bar i\in\mathbb{N}$ such that
	$$t^u_i=t^*, \qquad \text{for all} \,\, i\geq\bar i \quad \text{and} \quad t^u_i>t^*, \qquad \text{for all} \,\, i<\bar i. $$
	Thus, $t^*=t^u_{\bar i}=t_{\bar i-1}$, and hence
	$$f(\bm x+t_{\bar{i}-1} {\bm d})-f(\bm x)> -\beta_1\, t_{\bar{i}-1}\,  \lVert\bm g^*\rVert, $$
	yielding $t^*\notin \mathcal T$, which violates the fact  $t^*\in \mathcal T$.
\end{proof}	
Now, we are prepared to state the main result for Algorithm \ref{Line-Search}. Before it, we need to make the following semismooth assumption about the objective function $f$, which is commonly used in nonsmooth optimization \cite{kiwielbook,Mifflin}.

\begin{assumption}\label{Assumption}
	For any $\bm z, \bm d\in\mathbb{R}^n$ and sequences $\{\boldsymbol{\xi}_i\}_i\subset\mathbb{R}^n$ and $\{h_i\}_i\subset\mathbb{R}_+$ satisfying $h_i\downarrow 0$ as $i\to\infty$ and $\boldsymbol{\xi}_i\in\partial f(\bm z+h_i\bm d)$, one has
	$$\limsup_{i\to\infty} \boldsymbol{\xi}_i^T \bm d\geq \liminf_{i\to\infty} \frac{f(\bm z+ h_i\bm d)-f(\bm z)}{h_i}.$$	
\end{assumption}
The locally Lipschitz function $f$ which satisfies the above assumption is called \emph{weakly upper semismooth} \cite{Mifflin}. The class of weakly upper semismooth functions is quite broad. For example, convex, concave, and max- and min-type functions are weakly upper semismooth (for more details, see \cite{bagirov2020} and \cite{Mifflin}).
\begin{theorem}
	Suppose that Assumption \ref{Assumption} holds. Then Algorithm \ref{Line-Search} terminates after finitely many iterations.
\end{theorem}

\begin{proof}\label{T1}
	By indirect proof suppose that Algorithm \ref{Line-Search} does not terminate.	Let $\mathcal{I}$ be as defined in part (iii) of Lemma \ref{L2}. Then $\mathcal{I}$ is infinite and
	\begin{equation}\label{T1-1}
	f(\bm x+t_i {\bm d})-f(\bm x)>-\beta_1 t_i \lVert \bm g^*\rVert, \quad \text{for all} \,\, i\in\mathcal{I}.
	\end{equation}
	Moreover, in virtue of part (ii) of Lemma \ref{L2}, we have $t_i\to t^*$ as $i\to\infty$ with $t^*\in \mathcal T$, i.e.,
	\begin{equation}\label{T1-2}
	f(\bm x+t^* {\bm d})-f(\bm x)\leq-\beta_1 t^* \lVert \bm g^*\rVert.
	\end{equation}
	Combining \eqref{T1-1} and \eqref{T1-2}, one can write
	\begin{equation}\label{T1-3}
	f(\bm x+t_i {\bm d})-f(\bm x+t^* {\bm d})>-\beta_1 \lVert\bm g^*\rVert(t_i-t^*), \quad \text{for all}\,\, i\in\mathcal{I}.
	\end{equation}
	Let $h_i:=t_i-t^*>0$, for all $i\in\mathcal{I}$, and $\bm z:=\bm x+t^*{\bm d}$. Then \eqref{T1-3} is represented as
	\begin{equation}\label{T1-4}
	-\beta_1\lVert\bm g^*\rVert< \frac{f(\bm z+h_i {\bm d})-f(\bm z)}{h_i}, \quad \text{for all}\,\, i\in\mathcal{I}.
	\end{equation}
	Inequality \eqref{T1-4} along with the semismoothness hypotheses of Assumption \ref{Assumption} yields
	\begin{equation}\label{T1-5}
	-\beta_1\lVert\bm g^*\rVert\leq \liminf_{i\xrightarrow{\mathcal{I}}\infty} \frac{f(\bm z+h_i {\bm d})-f(\bm z)}{h_i}\leq \limsup_{i\xrightarrow{\mathcal{I}}\infty} \boldsymbol{\xi}_i^T {\bm d}.
	\end{equation}
	On the other hand, as the algorithm does not terminate by the third conditional block, it must be the case that
	$$\boldsymbol{\xi}_i^T {\bm d}<-\beta_2\lVert\bm g^*\rVert, \quad \text{for all} \,\, i\in\mathcal{I}. $$
	Therefore
	$$ \limsup_{i\xrightarrow{\mathcal{I}}\infty} \boldsymbol{\xi}_i^T {\bm d}\leq -\beta_2\lVert\bm g^*\rVert <-\beta_1\lVert\bm g^*\rVert,$$
	which contradicts \eqref{T1-5}.
\end{proof}

\section{Computation of a $(\delta, \mathcal{G}_\varepsilon(\bm x))$-stationary point} \label{Sec4}
In this section, for a given $\delta>0$ and $\varepsilon>0$, we employ the proposed line search procedure of Algorithm \ref{Line-Search} to develop an iterative process for finding a $(\delta, \mathcal{G}_\varepsilon(\bm x))$-stationary point. Such a process is presented in Algorithm~\ref{Alg2}.

\begin{algorithm}
	\caption{Computation of a $(\delta, \mathcal{G}_\varepsilon(\bm x))$-stationary point}
	\label{Alg2}
	\hspace*{\algorithmicindent}\textbf{Inputs:} Starting point $\bm x_0\in\mathbb{R}^n$, radius $\varepsilon\in(0,1)$, stationarity tolerance $\delta>0$.  \\
	\hspace*{\algorithmicindent}\textbf{Output:} A $(\delta, \mathcal{G}_\varepsilon(\bm x))$-stationary point $\bm x\in\mathbb{R}^n$.\\
	
	\hspace*{\algorithmicindent}\textbf{Function:} $\bm x$\,\,=\,\,\texttt{DG-SP}\,($\bm x_0, \varepsilon, \delta$)
	
	\begin{algorithmic}[1]
		\STATE{\textbf{Initialization:} Compute $\boldsymbol{\xi}_0\in\partial f(\bm x_0)$, set $\mathcal{G}_\varepsilon(\bm x_0):=\{\boldsymbol{\xi}_0\}$ and $k:=0$ ;}
		\WHILE{ true }
		\STATE{Set $\bm g^*_k:=\arg\min\{\lVert \bm g\lVert \,\, : \,\, \bm g\in\texttt{conv} \mathcal{G}_\varepsilon(\bm x_k)   \}$ ;}
		\IF{$\lVert \bm g^*_k\rVert\leq \delta,$}
		\RETURN {$\bm x_k$ as a $(\delta, \mathcal{G}_\varepsilon(\bm x_k))$-stationary point and \textbf{Stop} ; }
		\ENDIF
		\STATE{Compute the search direction ${\bm d}_k:=-\bm g^*_k/\lVert \bm g^*_k\rVert$ ; }
		\STATE{Set $\{\bm s_k, {\rm \bm I_k} \}:=$ \texttt{T-PLS} ($\varepsilon$, $\mathbf x_k$, ${\bm d}_k$)  ;}
		
		\IF{${\rm \bm I}_k=1,$}
		\STATE {Set $\bm x_{k+1}:=\bm s_k$ ;}
		\STATE{Compute $\boldsymbol{\xi}_{k+1}\in\partial f(\bm x_{k+1})$ ; }
		\STATE{Set $\mathcal{G}_\varepsilon(\bm x_{k+1}):=\{\boldsymbol{\xi}_{k+1} \}$ ; }
		\ENDIF
		\IF{${\rm \bm I}_k=0,$}
		\STATE {Set $\bm x_{k+1}:=\bm x_k$ and $\boldsymbol{\xi}_{k+1}:=\bm s_k$  ;}
		\STATE{Set $\mathcal{G}_\varepsilon(\bm x_{k+1}):= \mathcal{G}_\varepsilon(\bm x_{k})\cup\{\boldsymbol{\xi}_{k+1} \}$ ; }
		\ENDIF
		\STATE{Set $k:=k+1$ ;}
		
		\ENDWHILE
	\end{algorithmic}
	\hspace*{\algorithmicindent}\textbf{ End Function}
\end{algorithm}

Regarding Algorithm \ref{Alg2}, let
\begin{equation}\label{Index-set}
\mathcal{A}:=\{k\in\mathbb{N}_0 \,\, : \,\, {\rm \bm I}_k=1  \}.
\end{equation}

In the rest of this section, we aim to show that Algorithm \ref{Alg2} terminates after finite number of iterations.

In the following lemma,  $lev_\alpha(f):=\{\bm x\in\mathbb{R}^n \,\,: \,\, f(\bm x)\leq \alpha \}$ is the $\alpha$-sublevel set of the function $f$. Since $f$ is locally Lipschitz, $lev_\alpha(f)$ is closed, for each $\alpha\in\mathbb{R}$. Moreover, at iteration $k$ of Algorithm \ref{Alg2},  it is assumed that the line search procedure of Algorithm~\ref{Line-Search} terminates at the $i_k$-th iteration.
\begin{lemma}\label{L3}
	Suppose that  Assumption \ref{Assumption} holds and $lev_{f(\bm x_0)}(f)$ is bounded. If Algorithm \ref{Alg2} does not terminate, i.e., $k\to\infty$, then $\mathcal{A}$ is finite.
\end{lemma}
\begin{proof}
	Since $lev_{f(\bm x_0)}(f)$ is bounded and closed, we conclude
	\begin{equation}\label{L3-0}
	f^*:=\min\{f(\bm x) \,\, : \,\, \bm x\in\mathbb{R}^n \}> -\infty.
	\end{equation}
	By indirect proof, assume $\mathcal A$ is infinite. As Algorithm \ref{Alg2} does not terminate, one has
	\begin{equation}\label{L3-1}
	\lVert\bm g^*_k\rVert>\delta, \quad \text{for all}\,\, k.
	\end{equation}
	Furthermore, for any $k\in\mathcal{A}$, we have ${\rm \bm I}_k=1$ and hence
	\begin{equation}\label{L3-1'}
	f(\bm x_{k+1})-f(\bm x_k)=f(\bm s_k)-f(\bm x_k)\leq -\beta_1 \bar t_{i_k} \lVert\bm g^*_k\rVert, \quad \text{for all} \,\, k\in\mathcal A.
	\end{equation}
	By construction of Algorithm \ref{Line-Search}, we have $\bar t_{i_k}\geq \bar t>0$. Thus, in view of \eqref{L3-1} and \eqref{L3-1'}, one can write
	\begin{equation}\label{L3-2}
	f(\bm x_{k+1})-f(\bm x_k)\leq -\beta_1 \bar t \delta, \quad \text{for all} \,\, k\in\mathcal A.
	\end{equation}
	Moreover, for any $k\in\mathbb{N}_0\setminus\mathcal{A}$, we have ${\rm \bm I_k}=0$ and thus
	\begin{equation}\label{L3-3}
	f(\bm x_{k+1})=f(\bm x_k), \quad \text{for all} \,\, k\in\mathbb{N}_0\setminus\mathcal A.
	\end{equation}
	Using \eqref{L3-2} and \eqref{L3-3} inductively, for each $k\in\mathbb{N}_0$, we get
	\begin{equation}\label{L3-4}
	f(\bm x_{k+1})\leq f(\bm x_0)-\sum_{\substack{j\in\mathcal{A}\\ j\leq k+1}} \beta_1 \bar t\delta.
	\end{equation}
	Since $\mathcal{A}$ is infinite, $\sum_{\substack{j\in\mathcal{A}\\ j\leq k+1}} \beta_1 \bar t\delta\to\infty$ as $k\to\infty$. Therefore, \eqref{L3-4} implies $f(\bm x_k)\to-\infty$ as $k\to\infty$, which contradicts \eqref{L3-0}.
\end{proof}
Our  principal result about Algorithm \ref{Alg2} is stated in the next theorem.

\begin{theorem}\label{T2}
	Suppose that Assumption \ref{Assumption} holds and $lev_{f(\bm x_0)}(f)$ is bounded. Then Algorithm \ref{Alg2} terminates in a finite number of iterations.
\end{theorem}
\begin{proof}
	By indirect proof, assume that Algorithm \ref{Alg2} does not terminate, i.e., $k\to\infty$. Therefore
	\begin{equation}\label{T2-0}
	\lVert \bm g^*_k\lVert>\delta, \quad \text{for all} \,\, k.
	\end{equation}
	Let $\mathcal{A}$ be as defined in \eqref{Index-set}. By Lemma \ref{L3}, $\mathcal{A}$ is finite, and we denote  the largest index in $\mathcal{A}$ by $\bar k$ (in case $\mathcal{A}=\emptyset$, we set $\bar k:=0$). Let $\bar{\bm x}:=\bm x_{\bar k+1}$. Then, for any $k>\bar k$, we have ${\rm \bm I}_k=0$, and hence $\bm x_{k+1}=\bar{\bm x}$, for all $k>\bar k$. Moreover
	\begin{equation}\label{T2-1}
	\mathcal G_\varepsilon(\bm x_{k+1})= \mathcal{G}_\varepsilon(\bm x_k) \cup \{ \boldsymbol{\xi}_{k+1} \}, \quad \text{for all} \,\,  k>\bar k,
	\end{equation}
	in which, $\boldsymbol{\xi}_{k+1}$ satisfies
	\begin{equation*}
	\boldsymbol{\xi}_{k+1}^T{\bm d}_k\geq -\beta_2\lVert \bm g^*_k\lVert,
	\end{equation*}
	or equivalently (note ${\bm d}_k=-\bm g^*_k/\lVert \bm g^*_k\lVert$)
	\begin{equation}\label{T2-2}
	\boldsymbol{\xi}_{k+1}^T{\bm g^*_k}\leq \beta_2\lVert \bm g^*_k\lVert^2.
	\end{equation}
	We also note that $\boldsymbol{\xi}_{k+1}\in\partial_\varepsilon f(\bar{\bm x})$ and $\bm g^*_k\in\texttt{conv}\mathcal{G}_\varepsilon(\bar{\bm x})\subset\partial_\varepsilon f(\bar{\bm x})$, for all $k>\bar k$. Compactness of $\partial_\varepsilon f(\bar{\bm x})$ yields
	$$C_1:=\sup\{ \lVert\boldsymbol{\xi}\rVert\,\,:\,\, \boldsymbol{\xi}\in\partial_\varepsilon f(\bar{\bm x}) \}<\infty.$$
	Set $C_2:=\max\{C_1, \delta\}$. Thus
	\begin{equation}\label{T2-3}
	\lVert \boldsymbol{\xi}_{k+1}-\bm g^*_k\rVert \leq 2C_2, \quad \text{for all} \,\, k>\bar k.
	\end{equation}
	Next, for any $t\in(0,1)$ and $k>\bar k$, we have
	\begin{align}\label{T2-4}
	\lVert \bm g^*_{k+1}\rVert^2&\leq \lVert t \boldsymbol{\xi}_{k+1} + (1-t) \bm g^*_k\lVert^2\nonumber \\&
	=t^2\lVert \boldsymbol{\xi}_{k+1}-\bm g^*_k\lVert^2+2t(\bm{g}^*_k)^T (\boldsymbol{\xi}_{k+1}-\bm g^*_k)+\lVert \bm g^*_k\rVert^2.
	\end{align}
	In view of \eqref{T2-2} and \eqref{T2-3}, one can continue \eqref{T2-4} as
	\begin{align}
	\lVert \bm g^*_{k+1}\rVert^2&\leq 4t^2C_2^2+2t\beta_2\lVert \bm g^*_k\lVert^2-2t \lVert \bm g^*_k\lVert^2+ \lVert \bm g^*_k\lVert^2 \nonumber\\&
	= 4t^2C_2^2+ \left(1-2t(1-\beta_2)\right) \lVert \bm g^*_k\lVert^2 \nonumber \\&
	=:\psi(t),
	\end{align}
	for all $t\in(0,1)$. One can observe $t^*:=(1-\beta_2)\lVert \bm g^*_k\lVert^2/4C_2^2\in(0,1)$ minimizes $\psi(t)$ and
	\begin{equation*}
	\psi(t^*)=\left(1-\frac{(1-\beta_2)^2\lVert \bm g^*_k\lVert^2}{4C_2^2} \right) \lVert \bm g^*_k\lVert^2.
	\end{equation*}
	Using \eqref{T2-0}, the above equality implies
	\begin{equation}\label{T2-5}
	\psi(t^*)\leq\left(1-\frac{(1-\beta_2)^2\delta^2}{4C_2^2} \right) \lVert \bm g^*_k\lVert^2.
	\end{equation}
	Since $\delta\leq C_2$ and $\beta_2\in(0,1)$, we conclude $\sigma:=1-\frac{(1-\beta_2)^2\delta^2}{4C_2^2}\in(0,1)$. Now, \eqref{T2-4} and \eqref{T2-5}
	imply
	\begin{equation}\label{T2-6}
	0\leq \lVert \bm g^*_{k+1}\rVert^2 \leq \psi(t^*)\leq \sigma \lVert\bm g^*_k\lVert^2, \quad \text{for all} \,\, k>\bar k,
	\end{equation}
	which means that the sequence $\{\lVert \bm g^*_{k}\rVert^2 \}_{k>\bar k}$ is decreasing and bounded from below, and hence it converges. Assume $\{\lVert \bm g^*_{k}\rVert^2 \}\to A$ as $k\to\infty$. Letting $k$ approach infinity in inequality \eqref{T2-6}, we obtain $0\leq A\leq\sigma A$. Since $\sigma\in(0,1)$, we conclude $A=0$. Therefore $\{\lVert \bm g^*_{k}\rVert^2 \}\to 0$ as $k\to\infty$, which contradicts \eqref{T2-0}.
\end{proof}

\begin{remark}\label{R1}
	Regarding Algorithm \ref{Alg2}, if the number of consecutive iterations with ${\rm \bm I}_k=0$ is large, the size of $\mathcal{G}_\varepsilon(\bm x_k)$ increases as $k$ grows (See Line 16 of Algorithm \ref{Alg2}). This issue may pose some difficulty with the size of the subproblem which is solved in Line 3 of the algorithm. In such situations, the user can optionally employ an  adaptive reset strategy to efficiently control the size of subproblems. Such a strategy has been proposed in Appendix \ref{Appendix}.
\end{remark}

\section{Computation of a Clarke stationary point}\label{Sec5}
For the given sequences $\{\delta_\nu\}\downarrow 0$ and $\{\varepsilon_\nu \}\downarrow 0$, the main aim of this section is to obtain a Clarke stationary point through a sequence of $(\delta_\nu, \mathcal{G}_{\varepsilon_\nu}(\bm x_{\nu+1}))$-stationary points. Algorithm \ref{Alg3} which has a simple structure generates such a sequence.

\begin{algorithm}
	\caption{Computation of a Clarke stationary point}
	\label{Alg3}
	\hspace*{\algorithmicindent}\textbf{Inputs:} Starting point $\bm x_0\in\mathbb{R}^n$, positive sequences $\{\delta_\nu\}\downarrow 0$ and $\{\varepsilon_\nu \}\downarrow 0$, and optimality tolerance $\eta>0$. \\
	\hspace*{\algorithmicindent}\textbf{Output:} A point $\bm x\in\mathbb{R}^n$ as an approximation of a Clarke stationary point.\\
	
	
	\begin{algorithmic}[1]
		\STATE{\textbf{Initialization:} Set $\nu:=0$ ;}
		\WHILE{ true }
		\STATE{Set $\bm x_{\nu+1}:=\,\texttt{DG-SP}\,(\bm x_\nu, \varepsilon_\nu, \delta_\nu)$ ;}
		\IF{$ \delta_\nu\leq \eta,$ \rm{\textbf{and}} $ \varepsilon_\nu\leq \eta,$}
		\RETURN {$\bm x_\nu$ as an approximation of a Clarke stationary point and \textbf{Stop} ; }
		\ENDIF
		\STATE{Set $\nu:=\nu+1$ ;}
		
		\ENDWHILE
	\end{algorithmic}
	
\end{algorithm}

In order to study the asymptotic behavior of Algorithm \ref{Alg3}, we assume $\eta=0$. Thus, the algorithm generates the infinite sequence $\{\bm x_\nu\}_\nu$. In the following theorem, we prove that any accumulation point of the sequence $\{\bm x_\nu\}_\nu$ is Clarke stationary for objective  $f$.

\begin{theorem}\label{T3}
	Suppose that  Assumption \ref{Assumption} holds and $lev_{f(\bm x_0)}(f)$ is bounded. If $\eta=0$ in Algorithm \ref{Alg3}, then any accumulation point of the sequence $\{\bm x_\nu\}_\nu$ generated by this algorithm is Clarke stationary for $f$.
\end{theorem}
\begin{proof}
	For any $\nu\geq 0$, Algorithm \ref{Alg3} generates the $(\delta_\nu, \mathcal{G}_{\varepsilon_\nu}(\bm x_{\nu+1}))$-stationary point $\bm x_{\nu+1}$, i.e.,
	\begin{equation}\label{T3-1}
	\min\{\lVert \bm g\rVert \,\, : \,\, \bm g\in\texttt{conv} \mathcal{G}_{\varepsilon_\nu}(\bm x_{\nu+1}) \} \leq \delta_\nu, \quad \text{for all}\,\, \nu\geq 0.
	\end{equation}
	Since $\bm x_\nu\in lev_{f(\bm x_0)}(f)$, for all $\nu\geq 0$, boundedness of $lev_{f(\bm x_0)}(f)$ implies that the sequence $\{\bm x_\nu\}_\nu$ has at least one accumulation point, say $\bm x^*$. Thus, there exists $\mathcal{V}\subset\mathbb{N}_0$ such that $\bm x_\nu\xrightarrow{\mathcal{V}}\bm x^*$. Therefore, in view of \eqref{T3-1}, we have
	\begin{equation}\label{T3-2}
	\min\{\lVert \bm g\rVert \,\, : \,\, \bm g\in\texttt{conv} \mathcal{G}_{\varepsilon_\nu}(\bm x_{\nu+1}) \} \leq \delta_\nu, \quad \text{for all}\,\, \nu\in\mathcal{V}.
	\end{equation}
	Let $\omega>0$ be arbitrary. Since $\delta_\nu\downarrow 0$ as $\nu\to\infty$, there exists $\bar{\nu}\in\mathcal{V}$ sufficiently large such that $\delta_\nu<\omega$, for all $\nu\geq \bar{\nu}$. Then, it follows from \eqref{T3-2} that
	\begin{equation}\label{T3-3}
	\lVert \bm g^*_\nu\rVert:=\min\{\lVert \bm g\rVert \,\, : \,\, \bm g\in\texttt{conv} \mathcal{G}_{\varepsilon_\nu}(\bm x_{\nu+1}) \}< \omega, \quad \text{for all}\,\, \nu\geq\bar{\nu}, \,\, \nu\in\mathcal{V}.
	\end{equation}
	Therefore, the sequence $\{ \lVert \bm g^*_\nu\rVert\}_\nu$ is bounded, and without loss of generality, one may assume $\bm g^*_\nu\to \bm g^*$ as $\nu\xrightarrow{\mathcal{V}}\infty$. Now, the fact that $$\bm g^*_\nu\in\texttt{conv} \mathcal{G}_{\varepsilon_\nu}(\bm x_{\nu+1})\subset \partial_{\varepsilon_\nu} f(\bm x_{\nu+1})$$
	along with the upper semicontinuity of the map $\partial_\cdot f(\cdot)$ implies $\bm g^*\in\partial f(\bm x^*)$. Consequently
	$$ \min\{\lVert \bm g\rVert \,\, : \,\, \bm g\in \partial f(\bm x^*)\}\leq \omega. $$
	Since $\omega>0$ was arbitrary, we conclude $\bm 0\in \partial f(\bm x^*)$.
\end{proof}

\section{Numerical experiments}\label{Sec6}
In this section, we apply our method to a set of academic and semi-academic test problems and report the most important results. The proposed method is called \textbf{Subopt}. First, we consider a set of academic test problems and compare the efficiency of the proposed method with some well-known nonsmooth solvers. Next, several semi-academic problems are considered to show the applicability of the method in various contexts. The following experiments have been implemented in \textsc{Matlab} software (R2017b) on a machine with Intel Core i5 CPU 2.5 GHz and 6 GB RAM. Our choices for the parameters are as follows.

In the line search procedure of Algorithm \ref{Line-Search}, we set $\beta_1:=10^{-6}, \beta_2:=0.1, \bar{t}:=\varepsilon/2$, and the parameter $p$ is set to be $25$. To initialize the step length $t_i$, we set $t_0:=(\bar{t}+\varepsilon)/2$. In line 16 of Algorithm~\ref{Line-Search}, since $\zeta\in(0,0.5)$, one may choose ${t}_{i+1}$ as
$${t}_{i+1}:=\frac{t^u_{i+1}+t^l_{i+1}}{2}\in \left[ t^l_{i+1}+\zeta(t^u_{i+1}-t^l_{i+1}),\, t^u_{i+1}-\zeta(t^u_{i+1}-t^l_{i+1})   \right].$$
Regarding Algorithm \ref{Alg3}, the sequences $\{\delta_\nu \}\downarrow 0$ and $\{\varepsilon_\nu \}\downarrow 0$ were defined by $\delta_{\nu+1}:=0.5 \delta_\nu$ and $\varepsilon_{\nu+1}:=0.5\varepsilon_\nu$ with $\delta_0:=1$, $\varepsilon_0:=0.1.$

\subsection{Academic problems and alternative solvers}
Table \ref{Table1} provides a collection of nonsmooth convex and nonconvex test problems. The first five of these problems are convex and the rest are nonconvex. In this table, $f^*$ denotes a known (local) optimal value. Note that all of these academic test problems can be formulated with any number of variables.

\begin{table}[h]
	\centering
	\caption{List of test problems}\label{Table1}
	\resizebox{\textwidth}{!}{	
		\begin{tabular}{|lclllllcllcllcl|}
			\hline \rule{0pt}{3ex}
			&Problem &  &  & Name &  &  & Convex? &  &  & $f^*$ &  &  & Ref. &  \\
			\cline{2-2}\cline{5-5}\cline{8-8}\cline{11-11}\cline{14-14}
			& 1\rule{0pt}{3ex} &  &  & MAXL &  &  & Yes &  &  & 0 &  &  & \cite{techreport2} &  \\
			& 2 &  &  & L1HILB &  &  & Yes &  &  & 0 &  &  &  \cite{techreport2} &  \\
			& 3 &  &  & MAXQ &  &  & Yes &  &  & 0 &  &  & \cite{Limited-bundle} &  \\
			& 4 &  &  & MXHILB &  &  & Yes &  &  & 0 &  &  & \cite{Limited-bundle} &  \\
			& 5 &  &  & Chained CB3 II &  &  & Yes &  &  & $2(n-1)$ &  &  & \cite{Limited-bundle} &  \\
			& 6 &  &  & Active faces &  &  & No &  &  & 0 &  &  & \cite{phdthesis} &  \\
			& 7 &  &  & Brown function 2  &  &  & No &  &  & 0 &  &  & \cite{Limited-bundle} &  \\
			& 8 &  &  & Chained Mifflin 2 &  &  & No &  &  & varies &  &  & \cite{Limited-bundle} &  \\
			& 9 &  &  & Chained crescent I  &  &  & No &  &  & 0 &  &  & \cite{Limited-bundle} &  \\
			& 10 &  &  & Chained crescent II  &  &  & No &  &  & 0 &  &  & \cite{Limited-bundle} &  \\
			\hline
		\end{tabular}	
	}
	
\end{table}

Table \ref{Table2} describes the set of considered nonsmooth solvers. In this table, \textbf{GS} stands for the well-known original gradient sampling method, which is capable of minimizing both convex and nonconvex objectives \cite{Burke2005}.  \textbf{BTNC} is a variant of the well-known and efficient proximal Bundle-Trust (BT) method, which can solve both convex and nonconvex minimization problems \cite{BT-method}.  \textbf{SubG} is the classical subgradient method \cite{shorbook}. Although the convergence of this method was proved only for convex functions, there are empirical evidences to believe that it is able to minimize some types of nonconvex nonsmooth functions \cite{Bagirov2014}. Thus, we also applied this method to the considered set of nonconvex test problems using some  heuristic approaches to choose the off-line sequence of step length. We used the \textsc{Matlab} code of the \textbf{GS} algorithm which is freely available \cite{Burke2005}, and the other solvers have been implemented by the authors of this paper.

\begin{table}
	{
		\centering
		\caption{A list of nonsmooth solvers}\label{Table2}
		\resizebox{\textwidth}{!}{
			\begin{tabular}{|llllllcllcl|}
				\hline
				\multicolumn{2}{|c}{{\rule{0pt}{3ex}Name}} &  &   &  &  & \multicolumn{1}{l}{{Method}} &  &  & {Ref.} &  \\
				\cline{2-3}\cline{6-7}\cline{9-10}
				\multicolumn{1}{|c}{}\rule{0pt}{3ex} & \textbf{Subopt} &  &   &  &  & \multicolumn{1}{l}{Descent Subgradient} &  &  & {The current work} &  \\
				\multicolumn{1}{|c}{} & \textbf{GS} &  &   &  &  & \multicolumn{1}{l}{{Gradient Sampling}} &  &  & \cite{Burke2005} &  \\
				\multicolumn{1}{|c}{} & \textbf{BTNC} &  &   &  &  & \multicolumn{1}{l}{Proximal Bundle} &  &  & \cite{BT-method} &  \\
				\multicolumn{1}{|c}{} & \textbf{SubG} &  &   &  &  & \multicolumn{1}{l}{{Classical Subgradient}} &  &  & \cite{Bagirov2014} &  \\
				\hline		
			\end{tabular}
		}}
	\end{table}

	In this experiment, each problem was run using a single starting point randomly generated from $\mathcal{B}(\bm x_0, (\lVert \bm x_0\rVert+1)/n)$, where $\bm x_0$ was suggested in the literature. Moreover, we stopped an algorithm once the relative error
	$$E_k:=\frac{\lvert f(\bm x_k)-f^* \rvert}{\lvert f^* \rvert+1}$$
	drops below the prespecified tolerance $5\times10^{-4}$. We also limited the number of iterations to $10^4$.
	
	Figure \ref{Fig1} shows the performance profiles \cite{Dolan-More} based on the number of function and subgradient evaluations and elapsed CPU time, for $n=50$ and $100$.
	
	\begin{figure}
		\centering
		\subfloat{{\includegraphics[width=\textwidth]{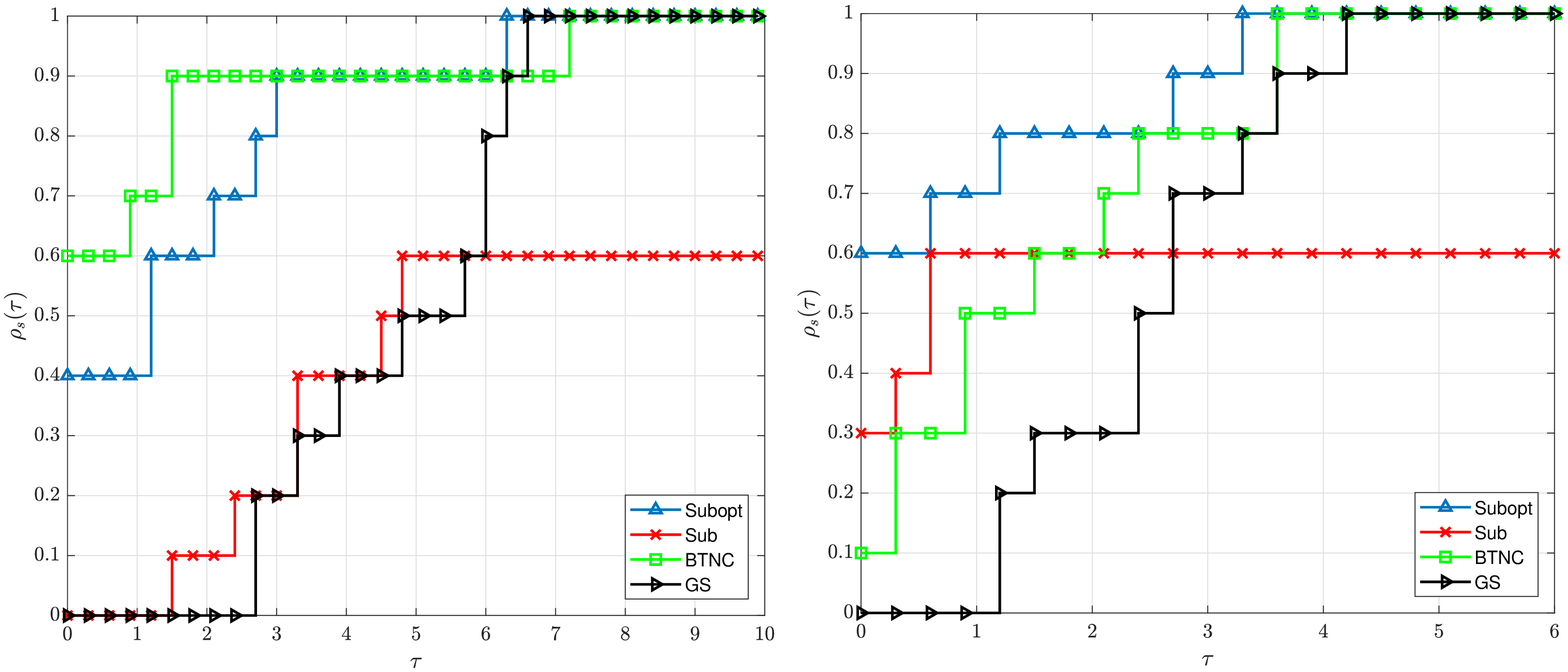} }}%
		
		\subfloat{{\includegraphics[width=\textwidth]{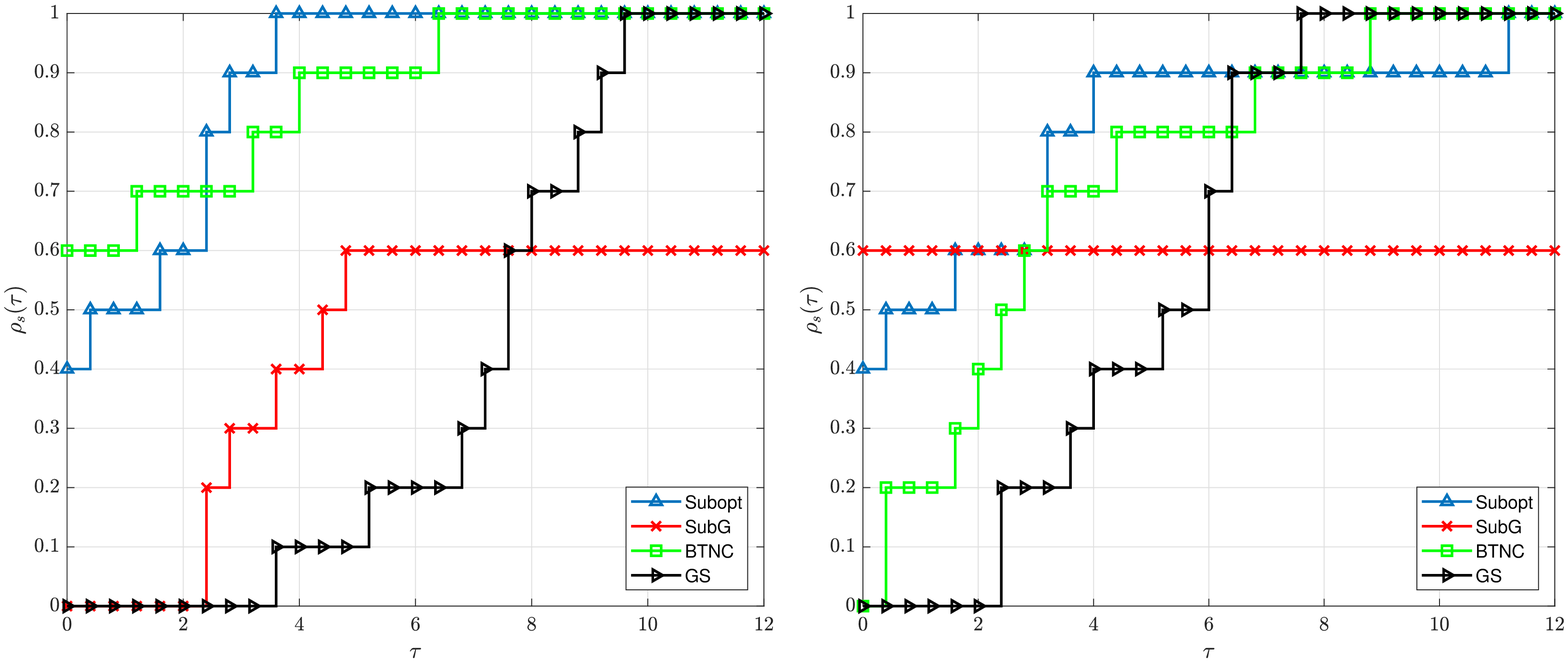} }}%
		\caption{Top: performance profiles based on the function and subgradient evaluations (Left) and elapsed CPU time (Right), for $n=50$. Bottom: the same for $n=100$.}
		\label{Fig1}
	\end{figure}
	
	As seen from Figure \ref{Fig1}, \textbf{Subopt}, \textbf{BTNC} and \textbf{GS} successfully reached the desired accuracy in all problems, while \textbf{SubG} has been successful in $60\%$ of problems with $n=50$ and $100$. In terms of function and subgradient evaluations, \textbf{BTNC} is superior to other solvers. Moreover, quite by a large margin, \textbf{GS} used more evaluations than the other solvers. This is due to
	the fact that this solver performs best when the size of the sample is set to $2n$. Furthermore, in majority of problems, \textbf{Subopt} consumed less CPU time than \textbf{BTNC}. This can be attributed to the fact that the structure of the quadratic subproblems in \textbf{Subopt} is simpler than \textbf{BTNC}. For $n=100$, although the \textbf{SubG} is the least robust solver, due to its simple structure, it consumed less CPU time than the other solvers in problems it successfully solved.
	
	\subsection{Data clustering}
	Assume $\mathscr{A}\subset\mathbb{R}^n$ is a finite set of data points, i.e.,
	$$\mathscr{A}=\{\bm a_1, \bm a_2, \ldots, \bm a_m  \}, \quad \text{where} \quad \bm a_i\in\mathbb{R}^n, \,\, i=1,\ldots,m.$$ For a given $\kappa\in\mathbb{N}$, our aim is to partition the set $\mathscr{A}$ into $\kappa$ subsets $\mathscr{A}_j, j=1,\ldots,\kappa$ such that
	\begin{itemize}
		\item [(i)] $\mathscr{A}_j\neq\emptyset, \quad j=1,\ldots,\kappa.$
		\item[(ii)] $\mathscr{A}_j\cap \mathscr{A}_{j'}=\emptyset,\quad j,j'=1,\ldots,\kappa, \, j\neq j'. $
		\item[(iii)] $\mathscr{A}=\cup_{j=1}^k \mathscr{A}_j.$
	\end{itemize}
	Such a problem is called \emph{hard clustering probelm} \cite{Bagirov2014}. Each cluster $\mathscr{A}_j$ is characterized by its center $\bm x_j$, $j=1,\ldots,\kappa$. Then, a data point $\bm a\in\mathscr{A}$ belongs to the cluster $\mathscr{A}_{\bar j}$ if
	$$\lVert \bm a-\bm x_{\bar j}\rVert=\min_{j=1,\ldots,\kappa} \, \lVert \bm a-\bm x_{j}\rVert, $$
	in which $\lVert \cdot \rVert$ is the squared Euclidean norm \cite{Bagirov2014}. The problem of finding the center points $\bm x_j, j=1,\ldots,\kappa,$ can be formulated by the following unconstrained nonsmooth optimization problem \cite{Bagirov2014}
	
	\begin{align}\label{Clustering_problem}
	\begin{split}
	&\min \, f_{\kappa}(\bm X)\\& \text{s.t.} \,\, \bm X=[\bm x_1, \bm x_2, \ldots, \bm x_{\kappa}]\in\mathbb{R}^{n\times\kappa},
	\end{split}
	\end{align}
	where
	$$f_{\kappa}(\bm X)= \frac{1}{m}\sum_{i=1}^{m}\min_{j=1,\ldots,\kappa} \lVert \bm a_i-\bm x_{j}\rVert. $$
	Note that, for any $\kappa>1$, this problem is nonsmooth and nonconvex. To create an instance of problem \eqref{Clustering_problem}, let $\mathscr{A}\subset\mathbb{R}^2$ be our data set containing ten thousands data points, i.e., $n=2$ and $m=10,000$. The top-left plot of Figure \ref{Fig2} depicts the data set $\mathscr{A}$. To solve this problem for different values of $\kappa$, we applied \textbf{Subopt} using a randomly generated starting point and optimality tolerance $\eta=10^{-8}$. The obtained results have been illustrated in Figure \ref{Fig2}, for $\kappa=2, 5, 10, 15$ and $20$.
	
	\begin{figure}
		\includegraphics[width=\textwidth]{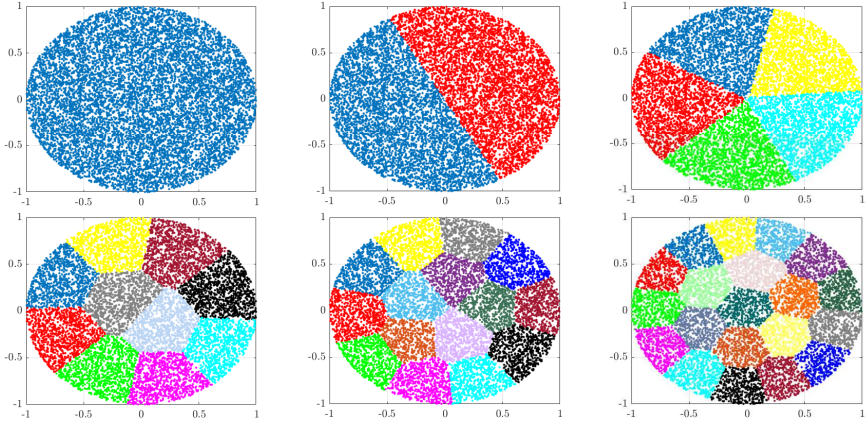}
		\caption{Illustration of the data set $\mathscr{A}$ (top-left) and clusters $\mathscr{A}_j$, for $\kappa=2$ (top-middle), $\kappa=5$ (top-right), $\kappa=10$ (bottom-left), $\kappa=15$ (bottom-middle), and $\kappa=20$ (bottom-right).}\label{Fig2}
	\end{figure}
	
	Table \ref{Table3} reports the computational cost of minimization of $f_\kappa(\bm X)$ for the considered values of $\kappa$. In this table, \#Fun and \#Sub denote the number of function and subgradient evaluations, respectively. Moreover, $f_\kappa(\bm X_{\rm end})$ is the value of the objective function of problem \eqref{Clustering_problem} at the last iteration.
	
	\begin{table}
		\centering
		\caption{Computational cost of minimization of $f_\kappa(\bm X)$ with $\eta=10^{-8}$. }\label{Table3}
		\resizebox{\textwidth}{!}{
			\begin{tabular}{|lcccllcccccccc|}
				\hline\rule{0pt}{3ex}
				& $\kappa$ &  & \#Fun &  &  & \#Sub &  &  & $f_{\kappa}(\bm X_{\rm end})$ &  &  & Time(s) &  \\
				\cline{2-2}\cline{4-4}\cline{7-7}\cline{10-10}\cline{13-13}
				& $2$\rule{0pt}{3ex} &  & 954 &  &  & 479 &  &  & 0.3212 &  &  & 51.23 &  \\
				& $5$ &  & 852 &  &  & 427 &  &  & 0.1249 &  &  & 42.12 &  \\
				& $10$ &  & 1200 &  &  & 600 &  &  & 0.0532 &  &  & 65.40 &  \\
				& $15$ &  & 633 &  &  & 317 &  &  & 0.0354 &  &  & 30.62 &  \\
				& $20$ &  & 1260 &  &  & 630 &  &  & 0.0260 &  &  & 62.43 &  \\
				\hline
			\end{tabular}
		}
	\end{table}
	
	\subsection{Chebyshev approximation by polynomials}
	Let $\mathcal{C}[a, b]$ be the space of real continuous functions on the closed interval $[a, b]$. One can equip this
	space with infinity norm, i.e., for each $f\in\mathcal{C}[a, b]$, define
	$$\lVert f\rVert_\infty:=\max \{\lvert f(x) \rvert \,\, : \,\, x\in[a,b] \}.$$
	Moreover, assume $\mathcal{P}_n$ is the space of polynomials of degree $m\leq n$. For a given $f\in\mathcal{C}[a, b]$ and $n\in\mathbb{N}_0$, our aim is to find $p^*\in\mathcal{P}_n$ which is the Chebyshev approximation (best uniform approximation) of $f$ in $\mathcal{P}_n$, i.e.,
	$$\min_{p\in\mathcal{P}_n} \lVert p-f\rVert_\infty=\lVert p^*-f\rVert_\infty. $$
	Let $\bm c=(c_n,\ldots, c_1, c_0)\in\mathbb{R}^{n+1}$ be the coefficients of a given polynomial $p\in\mathcal{P}_n$. Then, the above problem can be represented by
	\begin{equation}\label{Uniform-P}
	\min_{\bm c\in\mathbb{R}^{n+1}}\, h(\bm c),
	\end{equation}
	where
	$$h(\bm c):=\max_{x\in [a,b]}\, \lvert c_nx^n+\ldots+c_1x+c_0-f(x)\rvert.$$
	It is noted that problem \eqref{Uniform-P} is a nonsmooth convex minimization problem. To evaluate $h$ at a given $\bm c\in\mathbb{R}^{n+1}$, we have to solve a one-dimensional maximization problem. For this purpose,  we create a one-dimensional grid of the interval $[a,b]$ containing $2,000$ grid points. Next, we evaluate $\lvert c_nx^n+\ldots+c_1x+c_0-f(x)\rvert$ on the grid points and find the maximum value. The grid point at which the maximum occurs is used to initialize a local maximization method, which is based on the golden section and hyperbolic interpolation methods. Now consider $f(x):=\sin(2x)$ on the interval $[-\pi, \pi]$. We applied \textbf{Subopt} to find the Chebyshev approximation of $f$ in $\mathcal{P}_0$, $\mathcal{P}_1$, $\mathcal{P}_2$, and $\mathcal{P}_3$ using a randomly generated starting point and optimality tolerance $\eta=10^{-8}$.
	\begin{table}
		
		\caption{Results for the Chebyshev approximation of $\sin(2x)$  in $\mathcal{P}_0$, $\mathcal{P}_1$, $\mathcal{P}_2$, and $\mathcal{P}_3$ with  $\eta=10^{-8}$. }\label{Table4}
		\resizebox{\textwidth}{!}{

			\begin{tabular}{|lccclccccccccccllcl|}
				\hline \rule{0pt}{3ex}
				& $\mathcal{P}_n$ &  & $c^*_0$ &  & $c^*_1$ &  & $c^*_2$ &  & $c^*_3$ &  & \#Fun &  & \#Sub &  & $h_{\rm end}$ &  & Time(s) &  \\
				\cline{2-2}\cline{4-4}\cline{6-6}\cline{8-8}\cline{10-10}\cline{12-12}\cline{14-14}\cline{16-16}\cline{18-18}
				& $\mathcal{P}_0$\rule{0pt}{3ex} &  & -8.5197E-10 &  & - &  & - &  & - &  & 856 &  & 440 &  & 1.0000 &  & 5.57 &  \\
				& $\mathcal{P}_1$ &  & -6.3138E-11 &  & 2.4467E-9 &  & - &  & - &  & 1335 &  & 681 &  & 1.0000 &  & 9.21 &  \\
				& $\mathcal{P}_2$ &  & 1.0527E-10 &  & 1.0527E-10 &  & -1.0964E-9 &  & - &  & 1204 &  & 618 &  & 1.0000 &  & 6.64 &  \\
				& $\mathcal{P}_3$ &  & 2.1203E-10 &  & 0.1923 &  & 6.1199E-10 &  & -0.0472 &  & 3016 &  & 1463 &  & 0.8723 &  & 25.14 &  \\
				\hline
			\end{tabular}
		}	
		
	\end{table}
	
	The obtained results have been reported in Table \ref{Table4}. Based on these results, the Chebyshev approximation of $f$ in $\mathcal{P}_0$, $\mathcal{P}_1$, and $\mathcal{P}_2$ is the constant polynomial $p^*_1(x)\equiv0$, while the Chebyshev approximation of $f$ in $\mathcal{P}_3$ is
	$p^*_2(x)=-0.0472x^3+0.1923x$. The following \emph{alternation} theorem helps us to confirm the  results obtained by \textbf{Subopt}.

	\begin{theorem}[\cite{Poly-App}]\label{Alt-Theorem}
		$p^*\in\mathcal{P}_n$ is the best uniform approximation to $f\in C[a,b]$ if and only if the error function $e(x):=p^*(x)-f(x)$ takes consecutively the extremum value
		$\lVert p^*-f\lVert_\infty$ on $[a,b]$ with alternating signs at least $(n + 2)$ times.
	\end{theorem}
	
	Figure \ref{Fig3} indicates the error functions $e_1(x):=p^*_1(x)-\sin(2x)$ and $e_2(x):=p^*_2(x)-\sin(2x)$  on interval $[-\pi, \pi]$. As can be seen, $e_1(x)$ takes consecutively the extremum value $\lVert p^*_1-f\rVert_\infty=1$ with alternating sings at four points. Also, $e_2(x)$ takes consecutively the extremum value $\lVert p^*_2-f\rVert_\infty=0.8723$ with alternating sings at six points satisfying the requirements of Theorem \ref{Alt-Theorem}.

	\begin{figure}
		\includegraphics[width=\textwidth]{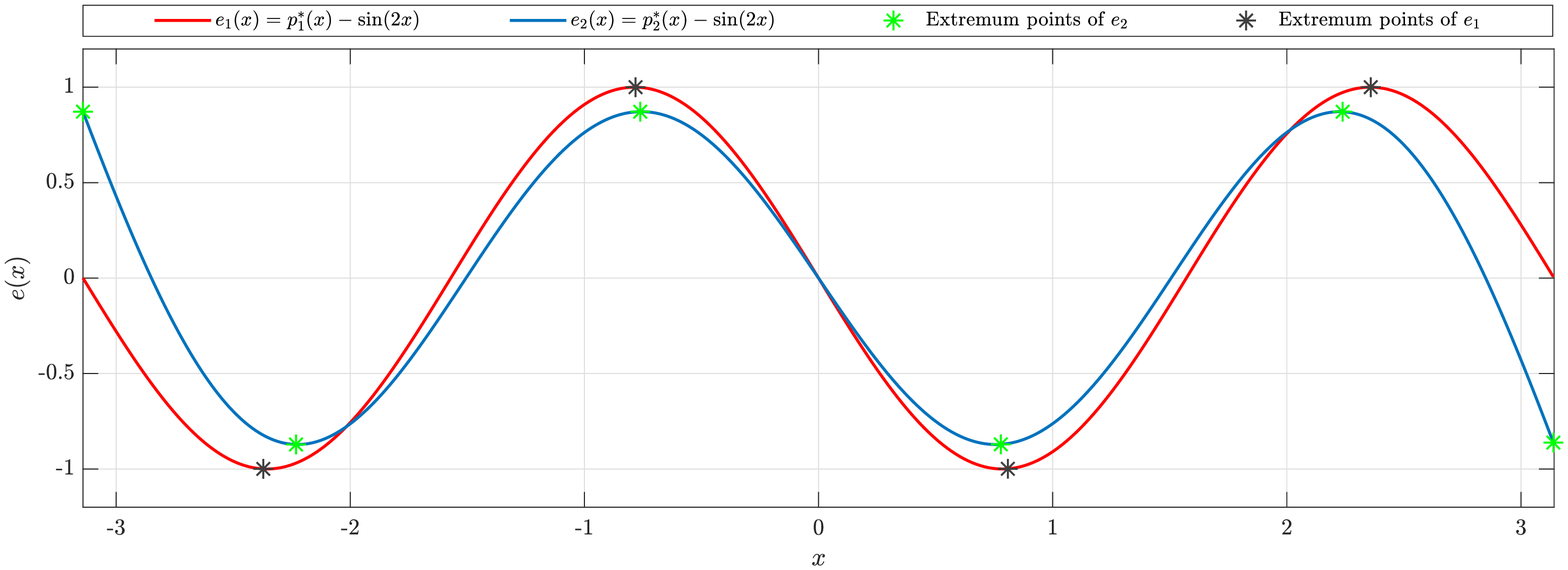}
		\caption{Error functions $e_1(x)$ and $e_2(x)$ on the interval $[-\pi, \pi]$.}\label{Fig3}
	\end{figure}

	\subsection{Minimization of eigenvalue products}
	
	For a positive semidefinite matrix $\bm A\in\mathbb{R}^{N\times N}$, 	we consider the problem
	\begin{align}\label{Prod}
	\begin{split}
	&\min_{\bm X\in\mathbb{R}^{N\times N}} \,\, \prod_{j=1}^{s} \lambda_j (\bm A \circ \bm X) \\&
	\text{s.t.} \,\,\, \bm X \succeq 0, \,\,\, \bm X_{i,i}=1, \quad i=1,\ldots,N,
	\end{split}
	\end{align}
	where $\bm A \circ \bm X$ is the componentwise product of matrices $\bm A$ and $\bm X$, and $\lambda_j (\bm A \circ \bm X)$ denotes the $j$-th largest eigenvalue of $\bm A \circ \bm X$. Problem  \eqref{Prod} was first presented in \cite{Burke2005} and solved by the original gradient sampling method. This problem is nonconvex, and its objective function is differentiable at $\bm X$ when $\bm X$ is positive definite and   $\lambda_s (\bm A \circ \bm X)> \lambda_{s+1} (\bm A \circ \bm X)$. Indeed, it has been shown in \cite{BFGS-Article} that the objective function of the problem is partly smooth. Alternatively, one can consider the following equivalent form of problem \eqref{Prod} as follows
	\begin{align}\label{Prod1}
	\begin{split}
	&\min_{\bm x\in\mathbb{R}^{n}} \,\, \prod_{j=1}^{s} \lambda_j (\bm A \circ Mat(\bm x)) \\&
	\text{s.t.} \,\,\, Mat(\bm x) \succeq 0,
	\end{split}
	\end{align}
	in which $n:=N(N-1)/2$ and
	$$Mat(\bm x):=\begin{bmatrix}
	1     & x_1    & x_2      & \ldots& x_{N-1} \\
	x_1   & 1      & x_N      &  \ldots& x_{2N-3}    \\
	x_2   & x_N    & 1        & \ldots & x_{3N-6} \\
	\vdots& \vdots & \vdots   & \ddots & \vdots\\
	x_{N-1}&x_{2N-3}& x_{3N-6}& \ldots  & 1\\
	\end{bmatrix} \in\mathbb{R}^{N\times N}
	.$$
	To handle the constraint of problem \eqref{Prod1}, following \cite{Burke2005}, we employ an exact penalty function, which leads to the following unconstrained minimization problem:
	\begin{equation}\label{Eigen-P}
	\min_{\bm x\in\mathbb{R}^{n}} \,\, \prod_{j=1}^{s} \lambda_j (\bm A \circ Mat(\bm x))-\mu\min\{0,\lambda_N(Mat(\bm x)) \}.
	\end{equation}
	In order to generate various instances of this problem, as suggested in \cite{Burke2005}, the matrices $\bm A$ are considered as the leading $N\times N$ submatrices of a $63\times63$ covariance matrix arising in an environmental application, which is freely available.
	
	\begin{table}
		\centering
		\caption{Results for minimization of eigenvalue products for various instances with $\eta=10^{-5}$.}\label{Table5}
		\resizebox{\textwidth}{!}{
			\begin{tabular}{|lcccccccccccccccccccc|}
				\hline\rule{0pt}{3ex}
				& $n$ &  &  & $N$ &  &  & $s$ &  &  & \#Fun &  &  & \#Sub &  &  & $f_{\rm end}$ &  &  & Time(s) &  \\
				\cline{2-2}\cline{5-5}\cline{8-8}\cline{11-11}\cline{14-14}\cline{17-17}\cline{20-20}
				& 1\rule{0pt}{3ex} &  &  & 2 &  &  & 1 &  &  & 248 &  &  & 130 &  &  & 0.4566 &  &  & 1.15 &  \\
				& 6 &  &  & 4 &  &  & 2 &  &  & 313 &  &  & 164 &  &  & 0.1580 &  &  & 1.37 &  \\
				& 15 &  &  & 6 &  &  & 3 &  &  & 3386 &  &  & 1690 &  &  & 0.0726 &  &  & 12.01 &  \\
				& 28 &  &  & 8 &  &  & 4 &  &  & 3055 &  &  & 1524 &  &  & 0.0327 &  &  & 10.25 &  \\
				& 45 &  &  & 10 &  &  & 5 &  &  & 6652 &  &  & 3326 &  &  & 0.0265 &  &  & 26.47 &  \\
				& 66 &  &  & 12 &  &  & 6 &  &  & 3146 &  &  & 1576 &  &  & 0.0106 &  &  & 16.40 &  \\
				& 91 &  &  & 14 &  &  & 7 &  &  & 3179 &  &  & 1547 &  &  & 0.0049 &  &  & 21.53 &  \\
				& 120 &  &  & 16 &  &  & 8 &  &  & 4053 &  &  & 2029 &  &  & 0.0035 &  &  & 44.71 &  \\
				\hline
			\end{tabular}
		}
	\end{table}
	
	Table \ref{Table5} shows the results obtained by \textbf{Subopt} when applied to different instances of the problem using  $\mu=100$, a randomly generated starting point, and optimality tolerance $\eta=10^{-5}$. In this table, $f_{\rm end}$ is the value of the objective function of problem \eqref{Eigen-P} at the last iteration.
	Comparing the results with those obtained by the original GS method, one can see that \textbf{Subopt} has provided better approximations of optimal values for different instances of the problem using optimality tolerance $\eta=10^{-5}$ (See Table 2 in \cite{Burke2005}).
	
	\section{Conclusion}\label{Sec7}
	We have developed a descent subgradient method for solving an unconstrained nonsmooth nonconvex optimization problem. To find an efficient descent direction, we proposed an iterative procedure to provide an inner approximation of the Goldstein $\varepsilon$-subdifferential. The least norm element of this approximation was considered as a search direction, and a new variant of Mifflin's line search was suggested for finding the next trial point. The finite convergence of the presented line search has been proved under the assumption that the objective function satisfies some semismoothness assumption. We  studied the global convergence of the method to a Clarke stationary point. Our numerical tests have confirmed that, the proposed subgradient algorithm is a big rival to GS and bundle type methods. Moreover, some semi-academic problems showed the applicability of the method in a variety of contexts.

%

\section*{Disclosure statement}

The authors report there are no competing interests to declare.

\section*{Funding}

This project was fully supported by the Iran National Science Foundation (INSF) under contract no. 99025023.

%
%
%
%
%
%

%
%

\bibliographystyle{tfnlm}
\bibliography{Maleknia}

\appendix
\section{ Reset strategy}\label{Appendix}
In this appendix, as stated in Remark \ref{R1}, we propose a reset strategy to efficiently control the size of quadratic subproblems in Algorithm \ref{Alg2}. Let $2<M\in\mathbb{N}$ be an upper bound for the number of elements in $\mathcal{G}_\varepsilon(\bm x_k)$. Once $\lvert \mathcal{G}_\varepsilon(\bm x_k) \rvert=M$ for some $k$, we adaptively discard almost inactive subgradients from $\mathcal{G}_\varepsilon(\bm x_k)$. To this end, assume we are at the $\bar k$-th iteration of Algorithm~\ref{Alg2} and
$$  \mathcal{G}_\varepsilon(\bm x_{\bar k})=\{\boldsymbol{\chi}_1,\boldsymbol{\chi}_2,\ldots,\boldsymbol{\chi}_{M-1}  \} \quad \text{and} \quad {\rm \bm I}_{\bar k}=0.  $$
Since ${\rm \bm I}_{\bar k}=0$, by construction of Algorithm \ref{Alg2}, we conclude $\lvert \mathcal{G}_\varepsilon(\bm x_{\bar k+1}) \rvert=M$. Thus, at this iteration, we eliminate some elements of $\mathcal{G}_\varepsilon(\bm x_{\bar k})$. Note that
$$\bm g^*_{\bar k}=\sum_{j=1}^{M-1} \lambda_j \boldsymbol{\chi}_j,$$
in which $\sum_{j=1}^{M-1} \lambda_j=1$ and $\lambda_j\geq 0$, for  all $j=1,\ldots, M-1$. Assume that $\lambda_{[j]}$ is the $j$-th largest element of $\boldsymbol{\lambda}=(\lambda_1,\lambda_2,\ldots,\lambda_{M-1})$. Now, for a prespecified weight $\theta\in(0,1]$, let $l\in\mathbb{N}$ be the smallest positive integer satisfying  $\sum_{j=1}^{l} \lambda_{[j]}\geq\theta$. Then, one can reduce $  \mathcal{G}_\varepsilon(\bm x_{\bar k})$ as
\begin{equation}\label{A1}
\mathcal{G}_\varepsilon(\bm x_{\bar k}) \leftarrow \{ \lambda_{[j]}: j=1,\ldots, l \}\cup \{ \bm g^*_{\bar k}\}.
\end{equation}
Consequently, the most active subgradients in forming $\bm g^*_{\bar k}$ are preserved and th rest are discarded. Moreover, due to a technical concern, the least norm element of $\texttt{conv} \mathcal{G}_\varepsilon(\bm x_{\bar k})$, namely $\bm g^*_{\bar k}$, has been appended to $\mathcal{G}_\varepsilon(\bm x_{\bar k})$.    The user should apply the elimination process \eqref{A1} right before Line 16 of Algorithm~\ref{Alg2} using the following conditional block:
\begin{align*}
`` &\textbf{if} \,\, \lvert \mathcal{G}_\varepsilon(\bm x_{k})\rvert=M-1 \,\, \textbf{then} \\&
\text{reduce } \mathcal{G}_\varepsilon(\bm x_{k})\, \text{as described in \eqref{A1}} ; \\&
\textbf{end if} \, "
\end{align*}
Eventually, we have to emphasize that this modification to Algorithm~\ref{Alg2} does not impair our proof of Theorem \eqref{T2}, because it is easy to see that
$$\boldsymbol{\xi}_{k+1}, \bm g^*_{k}\in \texttt{conv} \mathcal{G}_\varepsilon(\bm x_{k+1}), \quad \text{for all}\, k. $$
This explains why we have appended  $\bm g^*_{\bar k}$ to $\mathcal{G}_\varepsilon(\bm x_{\bar k})$ in \eqref{A1}.

\end{document}